\documentclass[12pt]{article}
\usepackage{makeidx}
\usepackage{graphicx}
\usepackage{latexsym}
\usepackage{amsmath}
\usepackage{amscd}
\usepackage{amssymb}
\usepackage{a4wide}
\author{St\'ephane Launois\thanks{This research was supported by a Marie Curie Intra-European
 Fellowship within the $6^{\mbox{th}}$ European Community Framework Programme held at the University of Edinburgh.}}

\title{On the automorphism groups of $q$-enveloping algebras of nilpotent Lie algebras.}
\date{ }

\newcommand{\para}{\addtocounter{subsubsection}{1}%
{\noindent{\it \thesubsubsection.\:}}}

\newcommand{\hc}{\mathcal{H}}
\newcommand{\spec}{\mathrm{Spec}}

\newcommand{\prim}{\mathrm{Prim}}

\newcommand{\aut}{\mathrm{Aut}}

\newcommand{\g}{\mathfrak{g}}

\newcommand{\comp}{\mathbb{C}}

\newcommand{\ideal}[1]{\langle {#1}\rangle}

\newcommand{\U}{U_q^+(\g)}
\newcommand{\Us}{U_q^+({\mathfrak sl}_3)}
\newcommand{\Ub}{U_q^+(\mathfrak{so}_5)}

\input{xy}
\xyoption{all}

\newenvironment{proof}{\begin{trivlist}\item[]{\it
Proof.}}{\hfill$\square$\end{trivlist}}

\begin{document}
\maketitle

\begin{abstract}

We investigate the automorphism group of the quantised enveloping algebra $\U$ of the positive nilpotent part of certain simple complex Lie algebras $\g$ in the case where the deformation parameter $q \in \mathbb{C}^*$ is not a root of unity. Studying its action on the set of minimal primitive ideals of $\U$ we compute this group in the cases where $\g= \mathfrak{sl}_3$ and $\g=\mathfrak{so}_5$ confirming a Conjecture of Andruskiewitsch and Dumas regarding the automorphism group of $\U$. In the case where $\g= \mathfrak{sl}_3$, we retrieve the description of the automorphism group of the quantum Heisenberg algebra that was obtained independently by Alev and Dumas, and Caldero. In the case where $\g=\mathfrak{so}_5$, the automorphism group of $\U$ was computed in \cite{launoisB2} by using previous results of Andruskiewitsch and Dumas. In this paper, we give a new (simpler) proof of the Conjecture of Andruskiewitsch and Dumas in the case where $\g=\mathfrak{so}_5$ based both on the original proof and on graded arguments developed in \cite{LaunoisLenagan2005} and \cite{launoislopes}. 
\end{abstract}

%%%%%%%%%%%%%%%%%%%%%%%%%%%%%%%%%%%%%%%%%%%%%%%%%%%%%%%%%%%%%%%%%%%%%%%%%%%%%%%%%%%%%%%%%%%%%%%%%%%%%%%%%%%%%%%%%%%%%%%
%%%%%%%%%%%%%%%%%%%%%%%%%%%%%%%%%%%%%%%%%%%%%%%%%%%%%%%%%%%%%%%%%%%%%%%%%%%%%%%%%%%%%%%%%%%%%%%%%%%%%%%%%%%%%%%%%%%%%%%

\newtheorem{theo}{Theorem}[section]
\newtheorem{defi}[theo]{Definition}
\newtheorem{defis}[theo]{Definitions}
\newtheorem{rem}[theo]{Remark}
\newtheorem{rems}[theo]{Remarks}
\newtheorem{nota}[theo]{Notation}
\newtheorem{notas}[theo]{Notations}
\newtheorem{prop}[theo]{Proposition}
\newtheorem{hypo}[theo]{Hypothese}
\newtheorem{lem}[theo]{Lemma}
\newtheorem{conv}[theo]{Convention}
\newtheorem{convs}[theo]{Conventions}
\newtheorem{cor}[theo]{Corollary}
\newtheorem{obs}[theo]{Observation}
\newtheorem{obss}[theo]{Observations}
\newtheorem{rap}[theo]{Recall}
\newtheorem{raps}[theo]{Rappels}

\newtheorem{paragraphe}{}[subsection]

%%%%%%%%%%%%%%%%%%%%%%%%%%%%%%%%%%%%%%%%%%%%%%%%%%%%%%%%%%%%%%%%%%%%%%%%%%%%%%%%%%%%%%%%%%%%%%%%%%%%%%%%%%%%%%%%%%%%%%%
%%%%%%%%%%%%%%%%%%%%%%%%%%%%%%%%%%%%%%%%%%%%%%%%%%%%%%%%%%%%%%%%%%%%%%%%%%%%%%%%%%%%%%%%%%%%%%%%%%%%%%%%%%%%%%%%%%%%%%%
\section*{Introduction}

In the classical situation, there are few results about the automorphism group of the enveloping algebra 
$U(\mathcal{L})$ of a Lie algebra $\mathcal{L}$ over $\comp$; except when $\dim\mathcal{L}\leq 2$, these groups are known to possess ``wild'' automorphisms and are far from being understood. For instance, this is the case 
when $\mathcal{L}$ is the three-dimensional abelian Lie algebra \cite{umirbaev}, when $\mathcal{L} = \mathfrak{sl}_{2}$ \cite{joseph} and when $\mathcal{L}$ is the three-dimensional Heisenberg Lie algebra \cite{alevclassique}.

In this paper we study the quantum situation. More precisely, we study the automorphism group of the quantised enveloping algebra $\U$ of the positive nilpotent part of a finite dimensional simple complex Lie algebra $\g$ in the case where the deformation parameter $q \in \mathbb{C}^*$ is not a root of unity. Although it is a common belief that quantum algebras are ''rigid'' and so should possess few symmetries, little is known about the automorphism group of $\U$. Indeed, until recently, this group was known only in the case where $\g=\mathfrak{sl}_3$ whereas  the structure of the automorphism group of the augmented form $\check{U}_{q} (\mathfrak{b}^{+})$, where $\mathfrak{b}^{+}$ is the 
positive Borel subalgebra of $\mathfrak{g}$, has been described in \cite{fleury} in the general case. 

The automorphism group of $\Us$ was computed independently by Alev-Dumas, \cite{alevdumasrigidite}, and Caldero, \cite{caldero}, who showed that 
$$ \aut(\Us) \simeq (\comp^*)^2 \rtimes S_2.$$
Recently, Andruskiewitsch and Dumas, \cite{andrusdumas} have obtained partial results on the automorphism group of $\Ub$. In view of their results and the description of $\aut(\Us)$, they have proposed the following conjecture. \\ \\
\\{\bf Conjecture (Andruskiewitsch-Dumas, \cite[Problem 1]{andrusdumas}):}
$$\aut(\U) \simeq (\comp^*)^{\mathrm{rk}(\g)} \rtimes \mathrm{autdiagr}(\g),$$
where $\mathrm{autdiagr}(\g)$ denotes the group of automorphisms of the Dynkin diagram of $\g$.\\

Recently we proved this conjecture in the case where $\g = \mathfrak{so}_5$, \cite{launoisB2}, and, in collaboration with Samuel Lopes, in the case where $\g = \mathfrak{sl}_4$, \cite{launoislopes}. The techniques in these two cases are very different. Our aim in this paper is to show how one can prove the Andruskiewitsch-Dumas Conjecture in the cases where $\g = \mathfrak{sl}_3$ and $\g = \mathfrak{so}_5$ by first studying the action of $\aut(\U)$ on the set of minimal primitive ideals of $\U$ - this was the main idea in \cite{launoisB2} -, and then using graded arguments as developed in \cite{LaunoisLenagan2005} and \cite{launoislopes}. This strategy leads us to a new (simpler) proof of the Andruskiewitsch-Dumas Conjecture in the case where $\g=\mathfrak{so}_5$.

Throughout this paper, $\mathbb{N}$ denotes the set of nonnegative integers, $\comp^*:=\comp \setminus \{0\}$ and $q$ is a nonzero complex number that is not a root of unity.

\section{Preliminaries.}

In this section, we present the $\hc$-stratification theory of Goodearl and Letzter for the positive part $\U$ 
of the quantised enveloping algebra of a simple finite-dimensional complex Lie algebra $\g$. In particular, we present a criterion (due to Goodearl and Letzter) that characterises the primitive ideals of $\U$ among its prime ideals. In the next section, we will use this criterion in order to describe the primitive spectrum of $\U$ in the cases where $\g=\mathfrak{sl}_3$ and $\g= \mathfrak{so}_5$. 

\subsection{Quantised enveloping algebras and their positive parts.}

Let $\g$ be a simple Lie $\comp$-algebra of rank $n$. We denote by
$\pi=\{\alpha_1,\dots,\alpha_n\}$ the set of simple roots
associated to a triangular decomposition $\g=\mathfrak{n}^- \oplus
\mathfrak{h} \oplus \mathfrak{n}^+$. Recall that $\pi$ is a basis of an
euclidean vector space $E$ over $\mathbb{R}$, whose inner product is
denoted by $(\mbox{ },\mbox{ })$ ($E$ is usually denoted by $
\mathfrak{h}_{\mathbb{R}}^*$ in Bourbaki). We denote by  $W$ the
Weyl group of $\g$, that is, the subgroup of the orthogonal group of
$E$ generated by the reflections $s_i:=s_{\alpha_i}$, for $i \in
\{1,\dots,n\}$, with reflecting hyperplanes $H_i:=\{\beta \in E \mid
(\beta,\alpha_i)=0\}$, $i \in \{1,\dots,n\}$. The length of $w \in W$ is denoted 
by $l(w)$. Further, we denote by $w_0$ the longest element
of $W$. We denote by  $R^+$ the set of positive roots and by $R$ the set of
 roots. Set $Q^+:=\mathbb{N} \alpha_1 \oplus \dots \oplus
 \mathbb{N} \alpha_n$ and $Q:=\mathbb{Z} \alpha_1 \oplus \dots \oplus
 \mathbb{Z} \alpha_n$. Finally, we denote by $A=(a_{ij}) \in M_n(\mathbb{Z})$ the Cartan
 matrix associated to these data. As $\g$ is simple, $a_{ij} \in \{0,-1,-2,-3\}$ for all $i \neq j$. 

Recall that the scalar product of two roots $(\alpha,\beta)$ is always
an integer. As in \cite{bg}, we assume that the short roots have
length $\sqrt{2}$. 
%(This will impose to be carefull when we will use
%\cite{bourbaki} since, for $\g$ of type $B_n$ or $F_4$, this condition
%is not satisfied: the short roots have length $1$, so that we will
%have to renormalise the roots.)

For all $i \in \{1,\dots,n \}$, set
$q_i:=q^{\frac{(\alpha_i,\alpha_i)}{2}}$ and 
$$\left[ \begin{array}{l} m \\ k \end{array} \right]_i:=
\frac{(q_i-q_i^{-1}) \dots
  (q_i^{m-1}-q_i^{1-m})(q_i^m-q_i^{-m})}{(q_i-q_i^{-1})\dots
  (q_i^k-q_i^{-k})(q_i-q_i^{-1})\dots (q_i^{m-k}-q_i^{k-m})} $$
for all integers $0 \leq  k \leq  m$. By convention, 
$$\left[ \begin{array}{l} m \\ 0 \end{array} \right]_i:=1.$$

The quantised enveloping algebra $U_q(\g)$ of $\g$ over $\comp$ associated to
the previous data is the $\comp$-algebra generated by the
indeterminates $E_1,\dots,E_n,F_1,\dots , F_n,K_1^{\pm 1}, \dots, K_n^{\pm 1}$ subject to the following relations:
$$K_i K_j =K_j K_i $$
$$ K_i E_j K_i^{-1}=q_i^{a_{ij}}E_j  \mbox{ and }  K_i F_j K_i^{-1}=q_i^{-a_{ij}}F_j$$
$$E_i F_j -F_jE_i=\delta_{ij} \frac{K_i-K_i^{-1}}{q_i-q_i^{-1}} $$
and the quantum Serre relations:
\begin{eqnarray}
\label{Serrequantique} 
\sum_{k=0}^{1-a_{ij}} (-1)^k  \left[ \begin{array}{c} 1-a_{ij} \\ k
 \end{array} \right]_i E_i^{1-a_{ij} -k} E_j E_i^k=0  \mbox{ } (i \neq  j)
\end{eqnarray}
and 
$$\sum_{k=0}^{1-a_{ij}} (-1)^k  \left[ \begin{array}{c} 1-a_{ij} \\ k
 \end{array} \right]_i F_i^{1-a_{ij} -k} F_j F_i^k=0  \mbox{ } (i \neq  j).$$

We refer the reader to \cite{bg,jantzen,josephbook} for more details on
this (Hopf) algebra. Further, as usual, we denote by $U_q^+(\g)$ (resp. $U_q^-(\g)$)  the
subalgebra of $U_q(\g)$ generated by $E_1,\dots,E_n$
(resp. $F_1,\dots,F_n$) and by $U^0$ the subalgebra of $U_q(\g)$ generated by 
$K_1^{\pm 1},\dots, K_n^{\pm 1}$. Moreover, for all $\alpha =a_1 \alpha_1 + \dots + a_n \alpha_n \in Q$, we set 
$$K_{\alpha}:= K_1^{a_1} \cdots K_n^{a_n}.$$
%$U_q(\bk^+)$ (resp. $U_q(\bk^-)$)  the
%subalgebra of $U_q(\g)$ generated by $E_1,\dots,E_n,K_1^{\pm 1},
%\dots, K_n^{\pm 1}$ (resp. $F_1,\dots,F_n,K_1^{\pm 1},
%\dots, K_n^{\pm 1}$). 

As in the classical case, there is a triangular decomposition as vector spaces:
$$ U_q^-(\g) \otimes U^0 \otimes \U \simeq U_q(\g).$$
In this paper we are concerned with the algebra $\U$ that admits the following presentation, see \cite[Theorem 4.21]{jantzen}. The algebra $\U$ is (isomorphic to) the $\comp$-algebra generated by $n$ indeterminates $E_1, \dots , E_n$ subject to the quantum Serre relations (\ref{Serrequantique}).

%To finish this paragraph, let us recall that $\U$ admits the following presentation, see \cite[Theorem 4.21]{jantzen}. 
\subsection{PBW-basis of $\U$.}
\label{sectionPBW}

To each reduced decomposition of the longest element $w_0$ of the Weyl group $W$ of $\g$, Lusztig has associated a PBW basis of $\U$, see for instance \cite[Chapter 37]{lusztigbook}, \cite[Chapter 8]{jantzen} or \cite[I.6.7]{bg}. The construction relates to a braid group action by automorphisms on $\U$. Let us first recall this action. For all $s \in \mathbb{N}$ and $i \in \{ 1, \dots , n\}$, we set 
$$[s]_i := \frac{q_i^s-q_i^{-s}}{q_i-q_i^{-1}} \ \ \mbox{ and } \ \ [s]_i ! :=[1]_i \dots [s-1]_i [s]_i.$$
As in \cite[I.6.7]{bg}, we denote by $\rm{T_i}$, for $ 1 \le i \leq n$, the automorphism of $\U$ defined by:  

$$ T_i(E_i)=-F_iK_i,$$
$$T_i(E_j)=\sum_{s=0}^{-a_{ij}} (-1)^{s-a_{ij}} q_i^{-s} E_i^{(-a_{ij}-s)} E_j E_i^{(s)}, \hspace{5mm} i \neq j$$
$$ T_i(F_i)=-K_i^{-1}E_i,$$
$$T_i(F_j)=\sum_{s=0}^{-a_{ij}} (-1)^{s-a_{ij}} q_i^{s} F_i^{(s)} F_j F_i^{(-a_{ij}-s)}, \hspace{5mm} i \neq j$$
$$T_i (K_{\alpha})=K_{s_i(\alpha)} , \hspace{5mm} \alpha \in Q,$$
where $E_i^{(s)}:= \frac{E_i^s}{[s]_i !}$ and $F_i^{(s)}:= \frac{F_i^s}{[s]_i !}$ for all 
$s \in \mathbb{N}$. It was proved by Lusztig that the automorphisms $T_i$ satisfy the braid relations, that is, if $s_is_j$ has order $m$ in $W$, then $$T_iT_jT_i \dots = T_j T_i T_j \dots ,$$
where there are exactly $m$ factors on each side of this equality.

 The automorphisms $T_i$ can be used in order to describe  PBW bases of $\U$ as follows. It is well-known that the length of $w_0$ is equal to the number $N$ of positive roots of $\g$. Let $s_{i_1} \cdots s_{i_N} $ be a reduced decomposition of $w_0$. For $k \in \{1, \dots, N\}$, we set $\beta_k:= s_{i_1} \cdots s_{i_{k-1}} (\alpha_{i_k})$. Then $\{\beta_1, \dots, \beta_N\}$ is exactly the set of positive roots of $\g$. Similarly, we define elements $E_{\beta_k}$ of $U_q(\g)$ by
$$E_{\beta_k}:= T_{i_1} \cdots T_{i_{k-1}} (E_{i_k}).$$
Note that the elements $E_{\beta_k}$ depend on the reduced decomposition of $w_0$. The following well-known results were proved by Lusztig and Levendorskii-Soibelman.

\begin{theo}[Lusztig and Levendorskii-Soibelman]
\label{theofond}
$ $
\begin{enumerate}
 \item For all $k \in \{ 1, \dots, N\}$, the element $E_{\beta_k}$ belongs to $\U$.
 \item If $\beta_k=\alpha_i$, then $E_{\beta_k}=E_i$.
 \item The monomials $E_{\beta_1}^{k_1} \cdots E_{\beta_N}^{k_N}$, with $k_1, \dots, k_N \in \mathbb{N}$, 
 form a linear basis of $\U$.
 \item For all $1 \leq i < j \leq N$, we have 
 $$E_{\beta_j} E_{\beta_i} -q^{-(\beta_i , \beta_j)} E_{\beta_i} E_{\beta_j}=
 \sum a_{k_{i+1},\dots,k_{j-1}} E_{\beta_{i+1}}^{k_{i+1}} \cdots E_{\beta_{j-1}}^{k_{j-1}},$$
 where each $a_{k_{i+1},\dots,k_{j-1}} $ belongs to $\comp$.
\end{enumerate}
\end{theo}

As a consequence of this result, $\U$ can be presented as a skew-polynomial algebra: 
$$\U= \comp [E_{\beta_1}] [E_{\beta_2}; \sigma_2 , \delta_2] \cdots [E_{\beta_N}; \sigma_N , \delta_N],$$
 where each $\sigma_i$ is a linear automorphism and each $\delta_i$ is a $\sigma_i$-derivation of the appropriate 
 subalgebra. In particular, $\U$ is a noetherian domain and its group of invertible elements is reduced to nonzero complex numbers.

\subsection{Prime and primitive spectra of $\U$.}
\label{sectionHstratification}

We denote by $\spec(\U)$ the set of prime ideals of $\U$. First, as $q$ is not a root of unity, it was proved by Ringel \cite{ringel} (see also \cite[Theorem 2.3]{GoodearlLetzter}) that, as in the classical situation, every prime ideal of $\U$ is completely prime.

In order to study the prime and primitive spectra of $\U$, we will use the stratification theory developed by Goodearl and Letzter. This theory allows the construction of a partition of these two sets by using the action of a suitable torus on $\U$. More precisely, the torus $\hc:=(\mathbb{C}^*)^n$ acts naturally by automorphisms on $\U$ via:
$$(h_1,\dots, h_n).E_i = h_i E_i \mbox{ for all } i \in \{1,\dots,n\}.$$
(It is easy to check that the quantum Serre relations are preserved by the group $\hc$.) Recall (see \cite[3.4.1]{andrusdumas}) that this action is rational. (We refer the reader to \cite[II.2.]{bg} for the defintion of a rational action.) A non-zero element $x$ of $\U$ is an {\it $\hc$-eigenvector of $\U$} if 
$h.x \in \mathbb{C}^{*}x$ for all $h \in \hc$. An ideal $I$ of $\U$ is {\it $\hc$-invariant} if $h.I=I$ for all $h \in \hc$. We denote by {\it $\hc$-$\spec(\U)$} the set of all $\hc$-invariant prime ideals of $\U$. It turns out that this is a finite set by a theorem of Goodearl and Letzter about iterated Ore extensions, see \cite[Proposition 4.2]{GoodLet}. In fact, one can be even more precise in our situation. Indeed, in \cite{gorelik}, Gorelik has also constructed a stratification of the prime spectrum of $\U$ using tools coming from representation theory. It turns out that her stratification coincides with the $\hc$-stratification, so that we deduce from \cite[Corollary 7.1.2]{gorelik} that 

\begin{prop}[Gorelik]
\label{hpremier}
 $\U$ has exactly $|W|$ $\hc$-invariant prime ideals. 
\end{prop}

The action of $\hc$ on $\U$ allows via the $\hc$-stratification
theory of Goodearl and Letzter (see \cite[II.2]{bg}) the construction
of a partition of $\spec(\U)$ as follows. If $J$ is an $\hc$-invariant 
prime ideal of $\U$, we denote by $\spec_J(\U)$ the
{\it $\hc$-stratum of $\spec(\U)$ associated to $J$}. Recall that
$\spec_J(\U):=\{ P \in \spec(\U) \mid \bigcap_{h \in \hc} h.P=J
\}$. 
Then the $\hc$-strata $\spec_J(\U)$ ($J \in \hc$-$\spec(\U)$) form a partition of $\spec(\U)$ (see \cite[II.2]{bg}):
$$\spec(\U) = \bigsqcup_{J \in \hc \mbox{-}\spec(\U)} \spec_J(\U).$$
Naturally, this partition induces a partition of the set $\prim(\U)$ of all (left) primitive ideals of $\U$ as follows. 
For all $J \in \hc$-$\spec(\U)$, we set $\prim_J(\U):=\spec_J(\U) \cap \prim(\U)$. Then it is obvious that 
the $\hc$-strata $\prim_J(\U)$ ($J \in \hc$-$\spec(\U)$) form a partition of $\prim(\U)$:
$$\prim(\U) = \bigsqcup_{J \in \hc \mbox{-}\spec(\U)} \prim_J(\U).$$
More interestingly, because of the finiteness of the set of $\hc$-invariant prime ideals of $\U$, the $\hc$-stratification theory provides a useful tool to recognise primitive ideals without having to find all its irreductible representations! Indeed, following previous works of Hodges-Levasseur, Joseph, and Brown-Goodearl, Goodearl and Letzter have characterised the primitive ideals of $\U$ as follows, see \cite[Corollary 2.7]{GoodLet} or \cite[Theorem II.8.4]{bg}. 

\begin{theo}[Goodearl-Letzter]
\label{theoprim}
$\prim_J(\U)$ $(J \in \hc$-$\spec(\U)$) coincides with those primes in $\spec_J(\U)$ that are maximal in $\spec_J(\U)$. 
\end{theo}

%To finish this paragraph, let us also mention that Goodearl and Lenagan have proved \cite[Theorem 4.8]{goodlenCAT} that $\U$ is catenary and that Tauvel's height formula holds in $\U$. 

\section{Automorphism group of $U_q^+(\g)$.}

In this section, we investigate the automorphism group of $\U$ viewed as the algebra generated by $n$ indeterminates $E_1, \dots , E_n$ subject to the quantum Serre relations. This algebra has some well-identified automorphisms. 
First, there are the so-called torus automorphisms; let $\mathcal{H}=(\comp^{*})^{n}$, where $n$ still denotes the rank of $\g$. As $\U$ is the $\comp$-algebra generated by $n$ indeterminates subject to the quantum Serre relations, it is easy to check that each $\bar{\lambda}=(\lambda_{1},\dots, , \lambda_{n})\in\mathcal{H}$ determines an algebra  automorphism $\phi_{\bar{\lambda}}$ of $\U$  with $\phi_{\bar {\lambda}}(E_{i})=\lambda_{i}E_{i}$ for $i \in \{1, \dots, n\}$, with inverse
$\phi^{-1}_{\bar{\lambda}}=\phi_{\bar{\lambda}^{-1}}$. Next, there are the so-called diagram automorphisms coming from the symmetries of the Dynkin diagram of $\g$. Namely, let $w$ be an automorphism of the Dynkin diagram of $\g$, that is, $w$ is an element of the symmetric group $S_n$ such that $(\alpha_{i}, \alpha_{j})=(\alpha_{w(i)}, \alpha_{w(j)})$ for all $i,j \in \{1, \dots ,n\}$. Then one defines an automorphism, also denoted $w$, of $\U$ by: $w(E_i)=E_{w(i)}$. Observe that 
$$\phi_{\bar{\lambda}} \circ w = w \circ \phi_{(\lambda_{w(1)},\dots, , \lambda_{w(n)})}.$$

We denote by $G$ the subgroup of $\aut(\U)$ generated by the torus automorphisms and the diagram automorphisms. Observe that 
$$G \simeq \mathcal{H} \rtimes \mathrm{autdiagr}(\g),$$
where $\mathrm{autdiagr}(\g)$ denotes the set of diagram automorphisms of $\g$.

The group $\aut(\Us)$ was computed independently by Alev and Dumas, see \cite[Proposition 2.3]{alevdumasrigidite} , and Caldero, see \cite[Proposition 4.4]{caldero}; their results show that, in the case where $\g=\mathfrak{sl}_3$, we have 
$$\aut(\Us) = G.$$
About ten years later, Andruskiewitsch and Dumas investigated the case where $\g=\mathfrak{so}_5$, see \cite{andrusdumas}. In this case, they obtained partial results that lead them to the following conjecture.\\ \\
\\{\bf Conjecture (Andruskiewitsch-Dumas, \cite[Problem 1]{andrusdumas}):}
$$\aut(\U) =G.$$
$ $

This conjecture was recently confirmed in two new cases: $\g=\mathfrak{so}_5$, \cite{launoisB2}, and $\g=\mathfrak{sl}_4$, \cite{launoislopes}. Our aim in this section 
is to show how one can use the action of the automorphism group of $\U$ on the primitive spectrum of this algebra in order to prove the Andruskiewitsch-Dumas Conjecture in the cases where $\g=\mathfrak{sl}_3$ and $\g=\mathfrak{so}_5$. 

\subsection{Normal elements of $\U$.}

Recall that an element $a$ of $\U$ is normal provided the left and right ideals generated by $a$ in $\U$ coincide, that is, if $$a\U = \U a.$$

In the sequel, we will use several times the following well-known  result concerning normal elements of $\U$.

\begin{lem}
\label{utile}
Let $u$ and $v$ be two nonzero normal elements of $\U$ such that $\ideal{u}=\ideal{v}$. Then
there exist $\lambda,\mu \in \comp^*$ such that $u=\lambda v $ and $v =\mu u$.
% \\ {\sf label:utile}
\end{lem}
\begin{proof}
It is obvious that units $\lambda,\mu$ exist with these
properties. However, the set of units of $\U$ is precisely $\comp^*$.
\end{proof}

\subsection{$\mathbb{N}$-grading on $\U$ and automorphisms.}\label{S:aut:ngr}

As the quantum Serre relations are homogeneous in the given  generators, there is an $\mathbb{N}$-grading on $\U$ obtained by assigning to  $E_{i}$ degree $1$. Let
\begin{equation}\U=\bigoplus_{i\in\mathbb{N}}\U_{i}\end{equation}
be the corresponding decomposition, with $\U_{i}$ the subspace of  homogeneous elements of degree $i$. In particular, $\U_{0}=\comp$ and $\U_ {1}$ is the $n$-dimensional space spanned by the generators $E_{1}, \dots, E_{n}$. For $t\in\mathbb{N}$ set $\U_{\geq t}=\bigoplus_{i\geq t}\U_{i}$  and define   $\U_{\leq t}$ similarly.

We say that the nonzero element $u\in\U$ has degree $t$, and write $ \mathrm{deg} (u)=t$, if $u\in \U_{\leq t}\setminus \U_{\leq t-1}$  (using the convention that $\U_{\leq -1}=\{ 0 \}$). As $\U$ is a domain, $\mathrm{deg} (uv)=\mathrm{deg} (u)+\mathrm {deg} (v)$ for $u, v\neq 0$.

\begin{defi}
Let $A=\bigoplus_{i\in\mathbb{N}}A_{i}$ be an $\mathbb{N}$-graded $\comp$-algebra with  $A_{0}=\comp$ which is generated as 
an algebra by $A_{1}=\comp x_{1}\oplus \cdots\oplus\comp x_{n}$. If for each $i\in\{ 1, \dots , n \}$  there exist 
$0\neq a\in A$ and a scalar $q_{i, a}\neq 1$ such that $x_{i}a=q_{i, a}ax_{i}$, then we say that $A$ is an 
$\mathbb{N}$-graded algebra with enough $q$-commutation relations.
\end{defi}

The algebra $\U$, endowed with the grading just defined, is a connected $\mathbb{N}$-graded algebra with enough $q$-commutation relations. Indeed, if $i \in \{1, \dots ,n \}$, then there exists $u \in \U$ such that $E_i u = q^{\bullet} u E_i$ where $\bullet$ is a nonzero integer. This can be proved as follows. As $\g$ is simple, there exists an index $j \in \{ 1, \dots ,n\}$ such that $j \neq i$ and $a_{ij} \neq 0$, that is, $a_{ij} \in \{-1,-2,-3\}$. Then $s_i s_j$ is a reduced expression in $W$, 
so that one can find a reduced expression of $w_0$ starting with $s_i s_j$, that is, one can write 
$$w_0 = s_i s_j s_{i_3} \dots s_{i_N}.$$ 
With respect to this reduced expression of $w_0$, we have with the notation of Section \ref{sectionPBW}:
$$\beta_1= \alpha_i \ \ \mbox{ and } \ \ \beta_2=s_i(\alpha_j)= \alpha_j - a_{ij}\alpha_i$$
Then it follows from Theorem \ref{theofond} that $E_{\beta_1}= E_i$, $E_{\beta_2}=E_{\alpha_j - a_{ij}\alpha_i}$ and 
$$E_i E_{\beta_2} =q^{(\alpha_i , \alpha_j - a_{ij}\alpha_i)}E_{\beta_2} E_i,$$
that is,
$$E_i E_{\beta_2} =q^{-(\alpha_i , \alpha_j)}E_{\beta_2} E_i.$$
As $a_{ij} \neq 0$, we have $(\alpha_i , \alpha_j) \neq 0$ and so $q^{-(\alpha_i , \alpha_j)} \neq 1$ 
since $q$ is not a root of unity. So we have just proved:

\begin{prop}
\label{propUngraded}
$\U$ is a connected $\mathbb{N}$-graded algebra with enough $q$-commutation relations.
\end{prop}
   
One of the advantages of $\mathbb{N}$-graded algebras with enough $q$-commutation relations is that any automorphism of such an algebra must conserve the valuation associated to the $\mathbb{N}$-graduation. More precisely, as $\U$ is a connected $\mathbb{N}$-graded algebra with enough $q$-commutation relations, we deduce from \cite{launoislopes} (see also \cite[Proposition 3.2]{LaunoisLenagan2005}) the following result.

\begin{cor}\label{C:aut:ngr}
Let $\sigma\in\aut(\U) $ and $x\in \U_{d}\setminus\{ 0 \}$. Then $\sigma (x) =y_{d}+y_{>d}$, for some
$y_{d}\in \U_{d}\setminus\{ 0 \}$ and $y_{>d}\in \U_{\geq d+1}$.
\end{cor}

\subsection{The case where $\g=\mathfrak{sl}_3$.}
\label{paraheisenberg}

In this section, we investigate the automorphism group of $\U$ in the case where $\g=\mathfrak{sl}_3$. In this case the Cartan matrix is $A= \left( \begin{array}{cc} 2 & -1 \\ -1 & 2 \end{array} \right)$, so that $ \Us$ is the $\mathbb{C}$-algebra generated by two indeterminates $E_1$ and $E_2$ subject to the following relations:
\begin{eqnarray}
 & & E_1^2 E_2   - (q + q^{-1} ) E_1 E_2 E_1 + E_2 E_1^2 = 0 \\
 & & E_2^2 E_1   - (q + q^{-1} ) E_2 E_1 E_2 + E_1 E_2^2 = 0
\end{eqnarray}
We often refer to this algebra as the quantum Heisenberg algebra, and sometimes we denote it by $\mathbb{H}$, as in the classical situation the enveloping algebra of $\mathfrak{sl}_3^+$ is the so-called Heisenberg algebra. 

We now make explicit a PBW basis of $\mathbb{H}$. The Weyl group of $\mathfrak{sl}_3$ is isomorphic to the symmetric group $S_3$, where $s_1$ is identified with the transposition $(1 \ 2)$ and $s_2$ is identified with $(2 \ 3)$. Its longest element is then $w_0= (1 3)$; it has two reduced decompositions: $w_0=s_1 s_2 s_1 =s_2 s_1 s_2$. Let us choose the reduced decomposition $s_1 s_2 s_1$ of $w_0$ in order to construct a PBW basis of $\Us$. According to Section \ref{sectionPBW}, this reduced decomposition leads to the following root vectors:
$$E_{\alpha_1}=E_1 \mbox{, } E_{\alpha_1 +\alpha_2}= T_1(E_2)=-E_1 E_2 +q^{-1} E_2 E_1 \mbox{ and } E_{\alpha_2}=T_1T_2(E_1)=E_2.$$
In order to simplify the notation, we set $E_3:= -E_1 E_2 +q^{-1} E_2 E_1$. Then, it follows from Theorem \ref{theofond} that
\begin{itemize}
\item The monomials $E_1^{k_1}E_3^{k_3}E_2^{k_2}$, with $k_1, k_2,k_3$ nonnegative integers, form a PBW-basis of $\Us$.   
\item $\mathbb{H}$ is the iterated Ore extension over $\mathbb{C}$ generated by the indeterminates $E_1,E_3,E_2$ subject to the following relations:
$$E_3 E_1 = q^{-1} E_1 E_3, \hspace{5mm} E_2 E_3 = q^{-1} E_3 E_2, \hspace{5mm} E_2 E_1 = q E_1 E_2 + q E_3.$$
In particular, $\mathbb{H}$ is a Noetherian domain, and its group of invertible elements is reduced to $\mathbb{C}^*$.
\item It follows from the previous commutation relations between the root vectors that $E_3$ is a normal element in $\mathbb{H}$, that is, $E_3 \mathbb{H}=\mathbb{H} E_3$.
\end{itemize}

In order to describe the prime and primitive spectra of $\mathbb{H}$, we need to introduce two other elements. 
The first one is the root vector $E_3':=T_2(E_1)=-E_2 E_1 + q^{-1} E_1 E_2$. This root vector would have appeared if we have  choosen the reduced decomposition $s_2 s_1 s_2$ of $w_0$ in order to construct a PBW basis  of $\mathbb{H}$. It follows from Theorem \ref{theofond} that $E_3'$ $q$-commutes with $E_1$ and $E_2$, so that 
$ E_3'$ is also a normal element of $\mathbb{H}$. Moreover, one can describe the centre of $\mathbb{H}$ using the two normal elements $E_3$ and $E_3'$. Indeed, in \cite[Corollaire 2.16]{alevdumasfractions}, Alev and Dumas have described the centre of $U_q^+({\mathfrak sl}_n)$; independently Caldero has described the centre of $\U$ for arbitrary $\g$, see \cite{calderocentre}. In our particular situation, their results show that the centre $Z(\mathbb{H})$ of $\mathbb{H}$ is a polynomial ring in one variable $Z(\mathbb{H})=\mathbb{C}[\Omega]$, where $\Omega =E_3 E_3'$.

%%%%%%%%%%%%%%%%%%%%%%%%%%%%%%%%%%%%%%%%%%%%%%%%%%%%%%%%%%%%%%%%%%%%%%%%%%%%%%%%%%%%%%%%%%%%%%%%%%%%%%%%%%%%%%%%%%%%%%%
%%%%%%%%%%%%%%%%%%%%%%%%%%%%%%%%%%%%%%%%%%%%%%%%%%%%%%%%%%%%%%%%%%%%%%%%%%%%%%%%%%%%%%%%%%
We are now in position to describe the prime and primitive spectra of $\mathbb{H}=U_q^+(sl(3))$; this was first achieved by Malliavin who obtained the following picture for the poset of prime ideals of $\mathbb{H}$, see \cite[Th\'eor\`eme 2.4]{malliavin}:

$$\xymatrix{ \langle \langle E_1,E_2-\beta \rangle \rangle \ar@{-}[rdd]&& \langle \langle  E_1,E_2 \rangle \rangle \ar@{-}[ldd]\ar@{-}[rdd] & & \langle \langle E_1 -\alpha , E_2 \rangle \rangle \ar@{-}[ldd] \\ 
& & & &  \\
 & \langle  E_1 \rangle \ar@{-}[ldd]\ar@{-}[rrrdd]  & &   \langle  E_2 \rangle \ar@{-}[llldd]\ar@{-}[rdd]  &  \\
& & & &  \\
\langle \langle E_3 \rangle \rangle \ar@{-}[rrdd] & & \langle \langle \Omega- \gamma  \rangle \rangle \ar@{-}[dd] & & \langle \langle  E_3' \rangle \rangle \ar@{-}[lldd]   \\
& & & &  \\
& & \langle  0 \rangle   & & }$$
where $\alpha,\beta,\gamma \in \mathbb{C}^*$.

Recall from Section \ref{sectionHstratification} that the torus $\hc=(\comp^*)^2$ acts on $\Us$ by automorphisms and that the $\hc$-stratification theory of Goodearl and Letzter constructs a partition of the prime spectrum of $\Us$ into so-called $\hc$-strata, this partition being indexed by the $\hc$-invariant prime ideals of $\Us$. Using this description of $\spec (\Us )$, it is easy to identify the $6=|W|$ $\hc$-invariant prime ideals of $\mathbb{H}$ and their corresponding $\hc$-strata. 
As $E_1$, $E_2$, $E_3$ and $E_3'$ are $\hc$-eigenvectors, the 6 $\hc$-invariant primes are: 
$$\ideal{0}, \ \ideal{E_3} , \ \ideal{E_3'} ,\ \ideal{E_1}, \ \ideal{E_2} \mbox{ and } \ideal{E_1,E_2}.$$ 
Moreover the corresponding $\hc$-strata are:
\\$\spec_{\ideal{0}}(\mathbb{H})= \left\{ \ideal{0} \right\} \cup \left\{ \ideal{\Omega - \gamma} \mid \gamma \in \comp^* \right\}$, 
\\$\spec_{\ideal{E_3}}(\mathbb{H})= \left\{ \ideal{E_3} \right\}$, 
\\$\spec_{\ideal{E_3'}}(\mathbb{H})= \left\{ \ideal{E_3'} \right\}$, 
\\$\spec_{\ideal{E_1}}(\mathbb{H})= \left\{ \ideal{E_1} \right\} \cup \left\{ \ideal{E_1,E_2 - \beta} \mid \beta \in \comp^* \right\}$, 
\\$\spec_{\ideal{E_2}}(\mathbb{H})= \left\{ \ideal{E_2} \right\} \cup \left\{ \ideal{E_1 - \alpha, E_2} \mid \alpha \in \comp^* \right\}$ 
\\and  $\spec_{\ideal{E_1,E_2}}(\mathbb{H})= \left\{ \ideal{E_1,E_2} \right\}$. 

We deduce from this description of the $\hc$-strata and the the fact that primitive ideals are exactly those primes that are maximal within their $\hc$-strata, see Theorem \ref{theoprim}, that the primitive ideals of $\Us$ are exactly those primes that appear in double brackets in the previous picture.

We now investigate the group of automorphisms of $\mathbb{H}=\Us$. In that case, the torus acting naturally on $\Us$ is $\hc=(\comp^*)^2$, there is only one non-trivial diagram automorphism $w$ that exchanges $E_1$ and $E_2$, and so the subgroup $G$ of $\aut (\Us)$ generated by the torus and diagram automorphisms is isomorphic to the semi-direct product $(\comp^*)^2 \rtimes S_2$. We want to prove that $\aut(\Us)=G$.  

In order to do this, we study the action of $\aut(\Us)$ on the set of primitive ideals that are not maximal. As there are only two of them, $\ideal{E_3}$ and $\ideal{E_3'}$, an automorphism of $\mathbb{H}$ will either fix them or permute them.

Let $\sigma$ be an automorphism of $\Us$. It follows from the previous observation that  
$$ \mbox{either } \sigma (\ideal{E_3})=\ideal{E_3} \mbox{ and } \sigma (\ideal{E_3'})=\ideal{E_3'},$$
$$\mbox{or } \sigma (\ideal{E_3})=\ideal{E_3'} \mbox{ and } \sigma (\ideal{E_3'})=\ideal{E_3}.$$
As it is clear that the diagram automorphism $w$ permutes the ideals $\ideal{E_3}$ and $\ideal{E_3'}$, we get that there exists an automorphism $g \in G$ such that 
$$  g \circ \sigma (\ideal{E_3})=\ideal{E_3} \mbox{ and } g \circ \sigma (\ideal{E_3'})=\ideal{E_3'}.$$
Then, as $E_3$ and $E_3'$ are normal, we deduce from Lemma \ref{utile} that there exist $\lambda,\lambda' \in \comp^*$ such that
$$g \circ \sigma(E_3)= \lambda E_3 \mbox{ and } g\circ \sigma (E_3')=\lambda ' E_3'.$$
In order to prove that $g \circ \sigma$ is an element of $G$, we now use the $\mathbb{N}$-graduation of $\Us$ introduced in Section \ref{S:aut:ngr}. With respect to this graduation, $E_1$ and $E_2$ are homogeneous of degree 1, and so $E_3$ and $E_3'$ are homogeneous of degree 2. Moreover, as $(q^{-2}-1)E_1E_2=E_3 + q^{-1} E_3'$, we deduce from the above discussion that 
$$g \circ \sigma (E_1 E_2)=\frac{1}{q^{-2}-1} \left( \lambda E_3 + q^{-1} \lambda' E_3' \right)$$
 has degree two. On the other hand, as $\Us$ is a connected $\mathbb{N}$-graded algebra with enough $q$-commutation relations by Proposition \ref{propUngraded}, it follows from Corollary \ref{C:aut:ngr}
that $\sigma(E_1)= a_1 E_1 +a_2 E_2 + u$ and $\sigma (E_2)=b_1 E_1 +b_2 E_2 +v$, where $(a_1,a_2), (b_1,b_2) \in \comp^2 \setminus \{(0,0)\}$, and $u,v \in \Us$ are linear combinations of homogeneous elements of degree greater than one. As $g \circ \sigma(E_1) . g \circ\sigma(E_2)$ has degree two, it is clear that $u=v=0$. To conclude that $g \circ \sigma \in G$, it just remains to prove that $a_2=0=b_1$. This can be easily shown by using the fact that $
g \circ \sigma (-E_1 E_2 +q^{-1} E_2 E_1) =g \circ \sigma (E_3) = \lambda E_3$; replacing $g \circ \sigma (E_1)$ and $g \circ \sigma (E_2)$ by $a_1 E_1 +a_2 E_2$ and $b_1 E_1 +b_2 E_2$ respectively, and then identifying the coefficients in the PBW basis, leads to 
$a_2=0=b_1$, as required. Hence we have just proved that $g \circ \sigma \in G$, so that $\sigma$ itself belongs to $G$ the subgroup of $\aut (\Us)$ generated by the torus and diagram automorphisms. Hence one can state the following result that confirms the Andruskiewitsch-Dumas Conjecture. 

\begin{prop}
$\aut(\Us) \simeq (\mathbb{C}^*)^2 \rtimes \mathrm{autdiagr}(\mathfrak{sl}_3)$
\end{prop}

This result was first obtained independently by Alev and Dumas, \cite[Proposition 2.3]{alevdumasrigidite}, and Caldero, \cite[Proposition 4.4]{caldero}, but using somehow different methods; they studied this automorphism group by looking at its action on the set of normal elements of $\Us$.

\subsection{The case where $\g=\mathfrak{so}_5$.}

In this section we investigate the automorphism group of $\U$ in the case where $\g=\mathfrak{so}_5$. 
In this case there are no diagram automorphisms, so that the Andruskiewitsch-Dumas Conjecture asks whether 
every automorphism of $\Ub$ is a torus automorphism. In \cite{launoisB2} we have proved their conjecture when $\g= \mathfrak{so}_5$. The aim of this section is to present a slightly different proof based both on the original proof and on the recent proof by S.A. Lopes and the author of the Andruskiewitsch-Dumas Conjecture in the case where $\g$ is of type $A_3$.

 In the case where $\g =\mathfrak{so}_5$, the Cartan matrix is $A= \left( \begin{array}{cc} 2 & -2 \\ -1 & 2 \end{array} \right)$, so that $ \Ub$ is the $\mathbb{C}$-algebra generated by two indeterminates $E_1$ and $E_2$ subject to the following relations:
\begin{eqnarray}
& & E_1^3 E_2   - (q^2 +1 + q^{-2} ) E_1^2 E_2 E_1 + (q^2 +1 + q^{-2} ) E_1 E_2 E_1^2 + E_2 E_1^3 = 0 \\
 & & E_2^2 E_1   - (q^2 + q^{-2} ) E_2 E_1 E_2 + E_1 E_2^2 = 0
\end{eqnarray}

We now make explicit a PBW basis of $\Ub$. The Weyl group of $\mathfrak{so}_5$ is isomorphic to the dihedral group $D(4)$. Its longest element is $w_0= - id$; it has two reduced decompositions: $w_0=s_1 s_2 s_1 s_2 =s_2 s_1 s_2 s_1$. Let us choose the reduced decomposition $s_1 s_2 s_1 s_2$ of $w_0$ in order to construct a PBW basis of $\Ub$. According to Section \ref{sectionPBW}, this reduced decomposition leads to the following root vectors:
$$E_{\alpha_1}=E_1 \mbox{, } E_{2\alpha_1 +\alpha_2}= T_1(E_2)=\frac{1}{(q+q^{-1})} \left( E_1^2 E_2 -q^{-1}(q+q^{-1}) E_1 E_2 E_1 +q^{-2} E_2 E_1^2 \right)  \mbox{, } $$
$$E_{\alpha_1+\alpha_2}= T_1 T_2 (E_1) = -E_1 E_2 + q^{-2} E_2 E_1 \mbox{ and } E_{\alpha_2}=T_1T_2T_1(E_2)=E_2.$$

In order to simplify the notation, we set $E_3:= -E_{\alpha_1+\alpha_2}$ and $E_4:=E_{2\alpha_1 +\alpha_2}$. Then, it follows from Theorem \ref{theofond} that
\begin{itemize}
\item The monomials $E_1^{k_1} E_4^{k_4} E_3^{k_3} E_2^{k_2}$, with $k_1, k_2,k_3,k_4$ nonnegative integers, form a PBW-basis of $\Ub$.   
\item $\Ub$ is the iterated Ore extension over $\mathbb{C}$ generated by the indeterminates $E_1$, $E_4$, $E_3$, $E_2$ subject to the following relations:
$$\begin{array}{lll}
E_4 E_1 = q^{-2} E_1 E_4 & & \\
E_3 E_1 = E_1 E_3 - (q+q^{-1}) E_4, \ \ &E_3 E_4 = q^{-2}E_4 E_3, \ \ & \\
E_2 E_1 = q^2 E_1 E_2 - q^2 E_3, & E_2 E_4 = E_4 E_2 - \frac{q^2-1}{q+q^{-1}} E_3^2, \ \ & E_2 E_3 =q^{-2} E_3 E_2.
\end{array}$$

In particular, $\Ub$ is a Noetherian domain, and its group of invertible elements is reduced to $\mathbb{C}^*$.
\end{itemize}

Before describing the automorphism group of $\Ub$, we first describe the centre and the primitive ideals of $\Ub$. 
The centre of $\U$ has been described in general by Caldero, \cite{calderocentre}. In the case where $\g=\mathfrak{so}_5$, his result shows that 
$Z(\Ub)$ is a polynomial algebra in two indeterminates 
$$Z(\Ub) =\comp [z,z'],$$
where 
$$z= (1-q^2)E_1 E_3 +q^2(q+q^{-1})E_4$$
and $$z'=  -(q^2-q^{-2})(q+q^{-1})E_4E_2 +q^2(q^2-1)E_3^2.$$ 

Recall from Section \ref{sectionHstratification} that the torus $\hc=(\comp^*)^2$ acts on $\Ub$ by automorphisms and that the $\hc$-stratification theory of Goodearl and Letzter constructs a partition of the prime spectrum of $\Ub$ into so-called $\hc$-strata, this partition being indexed by the $8=|W|$ $\hc$-invariant prime ideals of $\Ub$.
In \cite{launoisB2}, we have described these eight $\hc$-strata. More precisely, we have obtained the following picture for the poset $\spec(\Ub)$,

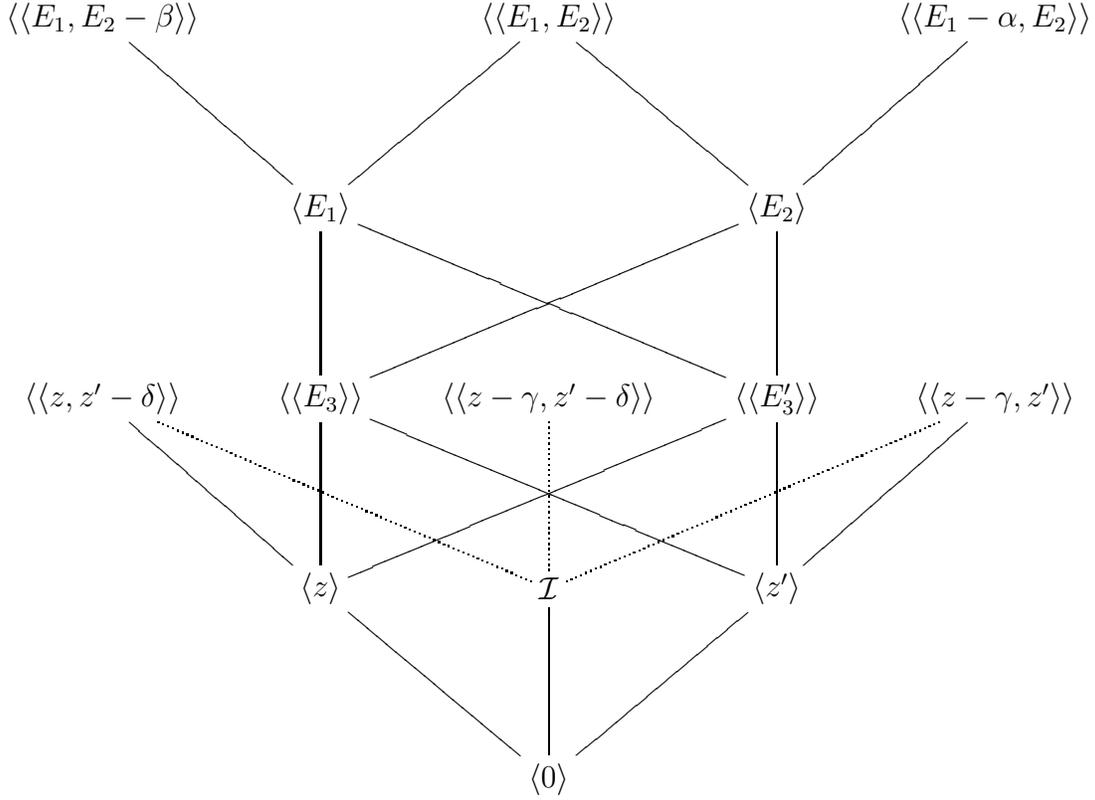
$$\xymatrix{ \langle \langle E_1,E_2-\beta \rangle \rangle \ar@{-}[rdd]&& \langle \langle E_1,E_2 \rangle \rangle \ar@{-}[ldd]\ar@{-}[rdd] & & \langle \langle E_1 -\alpha , E_2 \rangle \rangle \ar@{-}[ldd] \\ 
& & & &  \\
 & \langle  E_1 \rangle \ar@{-}[dd]\ar@{-}[rrdd]  & &   \langle  E_2 \rangle \ar@{-}[lldd]\ar@{-}[dd]  &  \\
& & & &  \\
 \langle \langle z,z'- \delta \rangle \rangle \ar@{-}[rdd] \ar@{.}[rrdd]& \langle \langle E_3 \rangle \rangle \ar@{-}[rrdd] \ar@{-}[dd]& \langle \langle z-\gamma,z'- \delta \rangle \rangle \ar@{.}[dd] & \langle \langle E_3' \rangle \rangle \ar@{-}[lldd] \ar@{-}[dd] &  \langle \langle z-\gamma ,z' \rangle \rangle \ar@{-}[ldd] 
\ar@{.}[lldd] \\
& & & &  \\
& \langle  z \rangle \ar@{-}[rdd]&   \mathcal{I} \ar@{-}[dd]& \langle  z' \rangle \ar@{-}[ldd]& \\
& & & &  \\
&  &  \langle  0 \rangle &  &
 }$$
where $\alpha,\beta,\gamma, \delta \in \mathbb{C}^*$, $E_3':=E_1 E_2 - q^{2} E_2 E_1$  and $$\mathcal{I}=\{ \ideal{P(z,z')} \ |\ P \mbox{ is a unitary irreductible polynomial of } \comp[z,z'], \ P \neq z , z' \}.$$

As the primitive ideals are those primes that 
are maximal in their $\hc$-strata, see Theorem \ref{theoprim}, we deduced from this description of the prime spectrum that the primitive ideals of $\Ub$ are the following: 
\\$\bullet$ $\ideal{z-\alpha,z'-\beta}$ with $(\alpha,\beta) \in \comp^2 \setminus \{(0,0) \}$.
\\$\bullet$ $\ideal{E_3}$ and $\ideal{E_3'}$.
\\$\bullet$ $\ideal{E_1- \alpha ,E_2-\beta}$ with $(\alpha,\beta) \in \comp^2 $ such that $\alpha \beta =0$.
\\(They correspond to the ``double brackets'' prime ideals in the above picture.)

Among them, two only are not maximal, $\ideal{E_3}$ and $\ideal{E_3'}$. Unfortunately, as $E_3$ and $E_3'$ are not normal in $\Ub$, one cannot easily obtain information using the fact that any automorphism of $\Ub$ will either preserve or exchange these two prime ideals. Rather than using this observation, we will use the action of $\aut (\Ub)$ on the set of maximal ideals of height two. Because of the previous description of the primitive spectrum of $\Ub$, the height two maximal ideals in $\Ub$ are those $\ideal{z-\alpha,z'-\beta}$ with $(\alpha,\beta) \in \comp^2 \setminus \{(0,0) \}$. In \cite[Proposition 3.6]{launoisB2}, we have proved that the group of units of the factor algebra 
$\Ub / \ideal{z-\alpha,z'-\beta}$ is reduced to $\comp^*$ if and only if both $\alpha$ and $\beta $ are nonzero. 
Consequently, if $\sigma$ is an automorphism of $\Ub$ and $\alpha \in \comp^*$, we get that:
$$\sigma ( \ideal{z- \alpha, z'}) = \ideal{z- \alpha', z'} \mbox{ or } \ideal{z, z' - \beta '},$$
where $\alpha', \beta' \in \comp^*$.
Similarly, if $\sigma$ is an automorphism of $\Ub$ and $\beta \in \comp^*$, we get that:
\begin{eqnarray}
\label{discussion}
\sigma ( \ideal{z, z'-\beta}) = \ideal{z- \alpha', z'} \mbox{ or } \ideal{z, z' - \beta '},
\end{eqnarray}
where $\alpha', \beta' \in \comp^*$. 

We now use this information to prove that the action of $\aut(\Ub)$ on the centre of $\Ub$ is trivial. More precisely, we are now in position to prove the following result.

\begin{prop}
\label{propcentre}
Let $\sigma \in \aut (\Ub)$. There exist $\lambda, \lambda' \in \comp^*$ such that 
$$\sigma(z)=\lambda z \ \ \mbox{ and } \ \ \sigma(z') = \lambda' z'.$$ 
\end{prop}
\begin{proof}
We only prove the result for $z$. First, using the fact that $\Ub$ is noetherian, it is easy to show that, for any family $\{\beta_i\}_{i \in \mathbb{N}}$ of pairwise distinct nonzero complex numbers, we have:
  
$$ \ideal{z} = \bigcap_{i \in \mathbb{N}} P_{0,\beta_{i}} \ \mbox{ and } \ \ideal{z'} = \bigcap_{i \in \mathbb{N}} P_{\beta_{i},0},$$
where $P_{\alpha,\beta}:=\ideal{ z-\alpha, z' - \beta }$. Indeed, if the inclusion $$ \ideal{z} \subseteq I:=\bigcap_{i \in \mathbb{N}} P_{0,\beta_{i}}$$
is not an equality, then any $P_{0,\beta_{i}}$ is a minimal prime over $I$ for height reasons. As the $P_{0,\beta_{i}}$ are pairwise distinct, $I$ is a two-sided ideal of $\Ub$ with infinitely many prime ideals minimal over it. This contradicts the noetherianity of $\Ub$. Hence  $$ \ideal{z} = \bigcap_{i \in \mathbb{N}} P_{0,\beta_{i}} \ \mbox{ and } \ \ideal{z'} = \bigcap_{i \in \mathbb{N}} P_{\beta_{i},0},$$
and so
$$\sigma \left( \ideal{z} \right)  = \bigcap_{i \in \mathbb{N}} \sigma (P_{0,\beta_{i}}).$$

It follows from (\ref{discussion})  that, for all $i \in \mathbb{N}$, there exists $(\gamma_i,\delta_i) \neq (0,0)$ 
with $\gamma_i = 0$ or $\delta_i =0$  such that 
$$\sigma (P_{0,\beta_{i}})= P_{\gamma_i,\delta_i}.$$
$ $

Naturally, we can choose the family $\{\beta_i\}_{i \in \mathbb{N}}$ such that 
either $\gamma_i=0$ for all $i \in \mathbb{N}$, or $\delta_i=0$ for all $i \in \mathbb{N}$. Moreover, observe that, as the $\beta_i$ are pairwise distinct, so are the $\gamma_i$ or the $\delta_i$.

Hence, either 
$$\sigma \left( \langle z \rangle \right)  = \bigcap_{i \in \mathbb{N}} P_{\gamma_i,0},$$
or $$\sigma \left( \langle z \rangle \right)  = \bigcap_{i \in \mathbb{N}} P_{0,\delta_{i}},$$
that is, 
$$\mbox{either }  \ideal{\sigma (z)} =\sigma \left( \ideal{z} \right)  =  \ideal{z'}  \mbox{ or }\ideal{\sigma (z)} =\sigma \left( \ideal{z} \right)  =  \ideal{z}  .$$

As $z$, $\sigma (z)$ and $z'$ are all central, it follows from Lemma \ref{utile} that there exists $\lambda \in \comp^*$ such that 
either $\sigma(z)= \lambda z$ or $\sigma(z)= \lambda z'$. 

To conclude, it just remains to show that the second case cannot happen. In order to do this, we use a graded argument. Observe that, with respect to the $\mathbb{N}$-graduation of $\Ub$ defined in Section \ref{S:aut:ngr}, $z$ and $z'$ are homogeneous of  degree 3 and 4 respectively. Thus, if $\sigma(z)= \lambda z'$, then we would obtain a contradiction with the fact that every automorphism of $\Ub$ preserves the valuation, see Corollary \ref{C:aut:ngr}. Hence $\sigma(z)= \lambda z$, as desired. The corresponding result for $z'$ can be proved in a similar way, so we omit it.  
\end{proof}

Andruskiewitsch and Dumas, \cite[Proposition 3.3]{andrusdumas}, have proved that the subgroup of those automorphisms of $\Ub$ that stabilize $\ideal{z}$ is isomorphic to $(\comp^*)^2$. Thus, as we have just shown that every automorphism of $\Ub$ fixes $\ideal{z}$, we get that $\aut(\Ub)$ itself is isomorphic to $(\comp^*)^2$. This is the route that we have followed in \cite{launoisB2} in order to prove the Andruskiewitsch-Dumas Conjecture in the case where $\g = \mathfrak{so}_5$. Recently, with Samuel Lopes, we proved this Conjecture in the case where $\g=\mathfrak{sl}_4$ using different methods and in particular graded arguments. We are now using (similar) graded arguments to prove that every automorphism of $\Ub$ is a torus automorphism (witout using results of Andruskiewitsch and Dumas).

In the proof, we will need the following relation that is easily obtained by straightforward computations.

\begin{lem}
\label{lemmerelations}
 $(q^2 - 1 ) E_3 E_3' = (q^4-1) z E_2 + q^2 z'$.
\end{lem}

\begin{prop}
Let $\sigma $ be an automorphism of $\Ub$. Then there exist $a_1,b_2 \in \comp^*$ such that 
$$\sigma (E_1) = a_1 E_1 \ \mbox{ and } \ \sigma(E_2)=b_2 E_2.$$
\end{prop}
\begin{proof}
For all $i \in \{ 1, \dots, 4\}$, we set $d_i:=\mathrm{deg} (\sigma(E_i))$. We also set $d_3':=\mathrm{deg} (\sigma(E_3'))$. It follows from Corollary \ref{C:aut:ngr} 
that $d_1,d_2 \geq 1$, $d_3,d_3' \geq 2$ and $d_4 \geq 3$. First we prove that $d_1=d_2=1$.

Assume first that $d_1+d_3 >3$. As $z= (1-q^2)E_1 E_3 +q^2(q+q^{-1})E_4$ and $\sigma(z) = \lambda z$ with $\lambda \in \comp^*$ by Proposition \ref{propcentre}, we get:
\begin{eqnarray}
 \label{eqz}
\lambda z &= &(1-q^2)\sigma(E_1) \sigma( E_3) +q^2(q+q^{-1})\sigma(E_4).
\end{eqnarray}
Recall that $\mathrm{deg} (uv)=\mathrm{deg} (u)+\mathrm {deg} (v)$ for $u, v\neq 0$, as $\U$ is a domain. Thus, as $\mathrm{deg} (z) =3 < \mathrm{deg}(\sigma(E_1) \sigma( E_3))=d_1+d_3$, we deduce from (\ref{eqz}) that $d_1 +d_3 = d_4$. As $z'=  -(q^2-q^{-2})(q+q^{-1})E_4E_2 +q^2(q^2-1)E_3^2$ and $\mathrm{deg} (z') =4 < d_1+d_3+d_2=d_4+d_2= \mathrm{deg}(\sigma(E_4) \sigma( E_2))$, we get in a similar manner that $d_2 +d_4 = 2 d_3$. Thus $d_1 +d_2=d_3$. As $d_1 +d_3 >3$, this forces $d_3 > 2$ and so $d_3 +d_3' > 4$. Thus we deduce from Lemma \ref{lemmerelations} that $d_3 +d_3'= 3 +d_2$. Hence $d_1 +d_3' = 3$. 
As $d_1 \geq 1$ and $d_3' \geq 2$, this implies $d_1=1$ and $d_3'=2$. 

Thus we have just proved that $d_1=\mathrm{deg} (\sigma(E_1))=1$ and either $d_3=2$ or $d_3'=2$. To prove that 
$d_2=1$, we distinguish between these two cases.

If $d_3=2$, then as previously we deduce from the relation $z'=  -(q^2-q^{-2})(q+q^{-1})E_4E_2 +q^2(q^2-1)E_3^2$ 
that $d_2+d_4=4$, so that $d_2=1$, as desired.

If $d_3'=2$, then one can use the definition of $E_3'$ and the previous expression of $z'$ in order to prove that 
$z'= q^{-2} (q^2-1)E_3^{'2} + E_2 u$, where $u$ is a nonzero homogeneous element of $\Ub$ of degree 3. 
($u$ is nonzero since $\ideal{z'}$ is a completely prime ideal and $E_3' \notin \ideal{z'}$ for degree reasons.) As $d_3'=2$ and  $\mathrm{deg} (\sigma(z')) =4$, we get as previously that $d_2=1$.

To summarise, we have just proved that $\mathrm{deg} (\sigma(E_1))=1=\mathrm{deg} (\sigma(E_2))$, so that 
 $\sigma(E_1)= a_1 E_1 +a_2 E_2 $ and $\sigma (E_2)=b_1 E_1 +b_2 E_2 $, where $(a_1,a_2), (b_1,b_2) \in \comp^2 \setminus \{(0,0)\}$. To conclude that $a_2=b_1=0$, one can for instance use the fact that $\sigma(E_1)$ and $\sigma(E_2)$ must satisfy the quantum Serre relations. 
\end{proof}

We have just confirmed the Andruskiewitsch-Dumas Conjecture in the case where $\g=\mathfrak{so}_5$. 

\begin{theo}
Every automorphism of $\Ub$ is a torus automorphism, so that
$$\aut (\Ub) \simeq (\comp^*)^2.$$ 
\end{theo}

\subsection{Beyond these two cases.}

To finish this overview paper, let us mention that recently the Andruskiewitsch-Dumas Conjecture 
was confirmed by Samuel Lopes and the author, \cite{launoislopes}, in the case where $\g=\mathfrak{sl}_4$. The crucial step of the proof is to prove that, up to an element of $G$, every normal element of $U_q^+(\mathfrak{sl}_4)$ is fixed by every automorphism. This step was dealt with by first computing the Lie algebra of derivations of $U_q^+(\mathfrak{sl}_4)$, and this already requires a lot of computations! \\

%%%%%%%%%%%%%%%%%%%%%%%%%%%%%%%%%%%%%%%%%%%%%%%%%%%%%%%%%%%%%%%%%%%%%% %%%%%%%%%%%%%%%%%%%%%%%%%%%%%%%%%%%%%%%%%%%%%%%%%
%%%%%%%%%%%%%%%%%%%%%%%%%%%%%%%%%%%%%%%%%%%%%%%%%%%%%%%%%%%%%%%%%%%%%% %%%%%%%%%%%%%%%%%%%%%%%%%%%%%%%%%%%%%%%%%%%%%%%%%

\begin{flushleft}
\textbf{Acknowledgments.} I thank Jacques Alev, Fran\c cois Dumas, Tom Lenagan and Samuel Lopes for all the interesting conversations that we have shared on the topics of this paper. I also like to thank the organisers of the Workshop ''From Lie Algebras to Quantum Groups'' (and all the participants) for this wonderful meeting. Finally, I would like to express my gratitude for the hospitality received during my subsequent visit to the University of Porto, especially from Paula Carvalho Lomp, Christian Lomp and Samuel Lopes.
\end{flushleft}

%%%%%%%%%%%%%%%%%%%%% ADDRESSES %%%%%%%%%%%%%%%%%%%%%%%%%%%%

\vskip 1cm

%\newpage 

\noindent St\'ephane Launois:\\
Institute of Mathematics, Statistics and Actuarial Science,\\
University of Kent at Canterbury, CT2 7NF, UK.\\
Email: S.Launois@kent.ac.uk

%%%%%%%%%%%%%%%%%%%%% END OF ADDRESSES %%%%%%%%%%%%%%%%%%%%%


\begin{thebibliography}{MMMM} 

\bibitem{alevclassique} J. Alev, {\it Un automorphisme non mod\'er\'e de $U(g_3)$}, Comm. Algebra 14 (8), 1365-1378 (1986).

%\bibitem{alevchamarie} J. Alev and M. Chamarie, {\it Automorphismes et d\'erivations de quelques alg\`ebres quantiques}, Comm. in Algebra 20 (1992), 1787-1802.



\bibitem{alevdumasrigidite} J. Alev and F. Dumas, {\em Rigidit\'e des plongements des quotients primitifs minimaux de  $U_q(sl(2))$ dans l'alg\`ebre quantique de Weyl-Hayashi}, Nagoya Math. J. 143 (1996), 119-146.

\bibitem{alevdumasfractions} J. Alev and F. Dumas, {\it Sur le corps des fractions de certaines alg\`ebres quantiques}, J. Algebra 170 (1994), 229-265.

\bibitem{andrusdumas} N. Andruskiewitsch and F. Dumas, {\em On the automorphisms of $U_q^+(\mathfrak{g})$}, ArXiv:math.QA/0301066, to appear.

\bibitem{bg} K.A. Brown and K.R. Goodearl, Lectures on algebraic quantum groups. Advanced Courses in Mathematics-CRM Barcelona. Birkh\"auser Verlag, Basel, 2002.

\bibitem{bgtrans} K.A. Brown and K.R. Goodearl,  {\em Prime spectra of quantum semisimple groups},  Trans. Amer. Math. Soc. 348 (1996),  no. 6, 2465-2502. 

\bibitem{calderocentre} P. Caldero, {\em Sur le centre de $U_q({\mathfrak n}^+)$}, Beitr\" age Algebra Geom.  35  (1994),  no. 1, 13-24. 

\bibitem{caldero} P. Caldero, {\em Etude des $q$-commutations dans l'alg\`ebre $U_q(n^+)$}, J. Algebra 178 (1995), 444-457.

%\bibitem{Dixmier} J. Dixmier, {\em Alg\`ebres enveloppantes}, Gauthier-Villars, Paris, 1974.

\bibitem{fleury} O. Fleury, {\em Automorphismes de $U_q(b^+)$}, Beitr\"age Algebra Geom. 38 (1994), 13-24.

\bibitem{GoodearlLetzter} K.R. Goodearl and E.S. Letzter, {\em Prime factor algebras of the coordinate ring of quantum matrices}, Proc. Amer. Math. Soc. 121 (1994), no. 4, 1017-1025.

\bibitem{GoodLet} K.R. Goodearl and E.S. Letzter, {\em The Dixmier-Moeglin equivalence in quantum coordinate rings and quantized Weyl algebras}, Trans. Amer. Math. Soc.  352  (2000),  no. 3, 1381-1403. 

\bibitem{gorelik} M. Gorelik, {\em The prime and the primitive spectra of a quantum Bruhat cell translate}, J. Algebra  227  (2000),  no. 1, 211-253.

\bibitem{jantzen} J.C. Jantzen, Lectures on Quantum Groups, in:
  Grad. Stud. Math., Vol. 6, Amer. Math. Society, Providence, RI, 1996. 
  
\bibitem{joseph} A. Joseph, {\em A wild automorphism of $U(sl_2)$}, Math. Proc. Camb. Phil. Soc. (1976), 80, 61-65.

\bibitem{josephbook} A. Joseph, Quantum groups and their primitive
  ideals. Springer-Verlag, 29, Ergebnisse der Mathematik und ihrer Grenzgebiete, 1995.

\bibitem{launoisB2} S. Launois, {\em Primitive ideals and automorphism group of $U_q^+(B_2)$}, J. Algebra Appl. 6 (2007), no. 1, 21-47.

\bibitem{LaunoisLenagan2005} S. Launois and T.H. Lenagan, {\em Primitive ideals and automorphisms of quantum matrices}, Algebr. Represent. Theory 10, no. 4, 339-365, 2007.

\bibitem{launoislopes} S. Launois and S.A. Lopes, {\em Automorphisms and derivations of $U_q(sl_4^+)$}, J. Pure Appl. Algebra 211, no.1, 249-264, 2007.

\bibitem{lusztigbook} G. Lusztig, Introduction to quantum groups, Progress in Mathematics, 110, Birkh\"auser Boston, Inc., Boston, MA, 1993.

\bibitem{malliavin} M.P. Malliavin, {\em L'alg\`ebre d'Heisenberg quantique}, Bull. Sci. Math. 118 (1994), 511-537.

\bibitem{ringel} C.M. Ringel, {\em PBW-bases of quantum groups},
  J. Reine Angew. Math. 470 (1996), 51-88. 

%\bibitem{smith} M.K. Smith, {\it Automorphisms of Envelopping algebras}, Comm. in Algebra 11 (16), 1769-1802 (1983).

\bibitem{umirbaev} I.P. Shestakov and U.U. Umirbaev, {\em The tame and the wild automorphisms of polynomial rings in three variables}, J. Amer. Math. Soc. 17 (2004), no. 1, 197-227 















  
 
\end{thebibliography}
\end{document}